\documentclass[11pt]{amsart}

\usepackage{amsfonts,graphics,amsmath,amsthm,amsfonts,amscd,amssymb,amsmath,latexsym,multicol,
 mathrsfs}
\usepackage{epsfig,url}
\usepackage{flafter}
\usepackage{fancyhdr}
\usepackage{hyperref}
\usepackage{amsmath,bm}
\usepackage{tikz}
\usepackage[section]{placeins}
\hypersetup{colorlinks=true,colorlinks=true,linkcolor=blue,urlcolor=blue, citecolor=blue}
\makeatletter

\def\jobis#1{FF\fi
  \def\predicate{#1}%
  \edef\predicate{\expandafter\strip@prefix\meaning\predicate}%
  \edef\job{\jobname}%
  \ifx\job\predicate
}

\makeatother

\if\jobis{proposal}%
\else
\fi

 \usepackage[matrix, arrow]{xy}

\DeclareMathOperator{\cent}{center}

 \numberwithin{equation}{section}
 \numberwithin{footnote}{section}

 \newtheorem{thm}{Theorem}[section]

 \newtheorem{cor}[thm]{Corollary}
 \newtheorem{lem}[thm]{Lemma}

 \newtheorem{defn-prop}[thm]{Definition-Proposition}
 \newtheorem{defn-lem}[thm]{Definition-Lemma}
{%_%_%_%_% upright style; roman (non-italic) text
    \newtheoremstyle{upright}%
        {8pt plus2pt minus4pt}%
        {8pt plus2pt minus4pt}%
        {\upshape}%
        {}%
        {\bfseries\scshape}%
        {}%
        {1em}%
        {}%
\theoremstyle{upright}

 \newtheorem{defn}[thm]{Definition}

 \newtheorem{rem}[thm]{Remark}

}

 \newcommand{\N}{\mathbb N}

 \newcommand{\Q}{\mathbb Q}
 \newcommand{\R}{\mathbb R}

\title{Singularities on vertical $\epsilon$-log canonical Fano fibrations}
\author{Caucher Birkar}
\address{Caucher Birkar, Yau Mathematical Sciences Center,
Jingzhai, Tsinghua University, Haidian District,
Beijing, China, 100084}
\email{birkar@tsinghua.edu.cn}
\author{Bingyi Chen}
\address{Bingyi Chen, Department of Mathematics,
Sun Yat-sen University,
Guangzhou, China, 510275}
\email{chenby253@mail.sysu.edu.cn, chenby16@tsinghua.org.cn}
\begin{document}

\begin{abstract}
Given a Fano type log Calabi-Yau fibration $(X,B)\to Z$ with $(X,B)$ being $\epsilon$-lc, the first author in \cite{Bi23} proved that the generalised pair $(Z,B_Z+M_Z)$ given by the canonical bundle formula is generalised $\delta$-lc where $\delta>0$ depends only on $\epsilon$ and $\dim X-\dim Z$, which confirmed a conjecture of Shokurov. In this paper, we prove the above result under a weaker assumption. Instead of requiring $(X,B)$ to be $\epsilon$-lc, we assume that $(X,B)$ is {\color{black} $\epsilon$-lc vertically over $Z$}, that is, the log discrepancy of $E$ with respect to $(X,B)$ is $\geq \epsilon$ for any prime divisor $E$ over $X$ whose center on $X$ is vertical over $Z$.
\end{abstract}

\maketitle
%\tableofcontents

%%%%%%%%%%%%%%%%%%%%%%%%
%%%%%%%%%%%%%%%%%%%%%%%%%%

\section{Introduction}
We work over an algebraically closed field of characteristic zero. 

Given a contraction $f:X\to Z$, a natural and important problem in birational geometry is to relate the singularities on $X$ and those on $Z$. This problem appears frequently in inductive understanding of varieties. In \cite{Bi23} the first author proved a conjecture of Shokurov regarding singularities on Fano type contractions. It says that given $d\in\N$ and $\epsilon\in \R^{>0}$, there is $\delta\in \R^{>0}$ such that if $f:X\to Z$ is a Fano type contraction of relative dimension $d$ and $(X,B)$ is an $\epsilon$-lc pair with $K_X+B\sim_{\R} 0/Z$, then the discriminant b-divisor defined by the canonical bundle formula has coefficients $\leq 1-\delta$. A consequence of this result is that if $f:X\to Z$ is a Fano fibration of relative dimension $d$ where $X$ is $\epsilon$-lc, then (1) $Z$ is $\delta$-lc if $K_Z$ is $\Q$-Cartier and (2) multiplicities of the fibres of $f$ over codimension one points of $Z$ are bounded from above by $1/\delta$, where $\delta>0$ depends only on $d,\epsilon$.

{\color{black}
It was conjectured by Shokurov (see \cite[Conjecture 1.5]{HJL22}) that a local version of the above result also holds. That is, for any prime divisor $E$ over $Z$, to make the coefficient of $E$ in the discriminant b-divisor at most $1-\delta$, one only need to require that $a(F,X,B)\geq \epsilon$ for any prime divisor $F$ over $X$ with $f(\cent_X F)=\cent_Z E$, instead of requiring that $(X,B)$ is $\epsilon$-lc. As a consequence of this conjecture, to make all the coefficients in the discriminant b-divisor at most $1-\delta$, one only need to require that $a(F,X,B)\geq \epsilon$ for any prime divisor $F$ over $X$ with $f(\cent_X F)\neq Z$, i.e. $(X,B)$ is $\epsilon$-lc vertically over $Z$ (see Definition \ref{defn:vertical} below).  It is worth to mention that under the original condition that $(X,B)$ is $\epsilon$-lc, the general fiber $F$ is a Fano type {\color{black}variety} with an $\epsilon$-lc $\R$-complement, hence belongs to a bounded family by \cite{Bi19,Bi21}. This fact plays {\color{black}an} important role in the proofs in \cite{Bi23}. However, under the new condition that  $(X,B)$ is $\epsilon$-lc vertically over $Z$, the general fiber $F$ may not belong to a bounded family.

The main purpose of this paper is to give an affirmative answer to the above conjecture.
}

{\color{black}
\begin{defn}\label{defn:along}
Let $X$ be a normal variety with a closed subvariety $V\subseteq X$. We say a (generalised) pair on $X$ is \emph{(generalised) $\epsilon$-lc along $V$} if the log discrepancy of $E$ with respect to this (generalised) pair is $\geq \epsilon$ for any prime divisor $E$ over $X$ with $\cent_X(E)\subseteq V$.
\end{defn}

\begin{defn}\label{defn:vertical}
Let $f:X\to Z$ be a projective morphism between normal varieties. We say a (generalised) pair on $X$ is \emph{(generalised) $\epsilon$-lc vertically over $Z$} if the log discrepancy of $E$ with respect to this (generalised) pair is $\geq \epsilon$ for any prime divisor $E$ over $X$ with $f(\cent_X(E))\neq Z$.
\end{defn}

\iffalse
\begin{defn}
Let $(X,B)$ be a pair and $V\subseteq X$ be a closed subvariety. We say $(X,B)$ is \emph{$\epsilon$-lc along $V$}
if $a(E,X,B)\geq \epsilon$ for any prime divisor $E$ over $X$ with $\cent_X(E)\subseteq V$.
\end{defn}

\begin{defn}
Let $(X,B)$ be a pair and $f:X\to Z$ be a projective morphism. We say $(X,B)$ is \emph{$\epsilon$-lc vertically over $Z$} if $a(E,X,B)\geq \epsilon$ for any prime divisor $E$ over $X$ with $f(\cent_X(E))\neq Z$.
\end{defn}

The above two definitions for generalised pairs are similar. 
\fi
}

\begin{thm}\label{thm:main}
Let $d\in \N$ and let $\epsilon\in \R^{>0}$. Then there is $\delta\in \R^{>0}$ depending only on $d,\epsilon$ satisfying the following. Let $(X,B)$ be a pair, $f:X\to Z$ be a contraction and $V\subseteq Z$ be a closed subvariety such that 
\begin{itemize}
  \item $\dim X-\dim Z=d$,
  \item $K_X+B\sim_{\R} 0/Z$,
  \item $X$ is of Fano type over $Z$, and
  \item $(X,B)$ is $\epsilon$-lc along $f^{-1} V$.
\end{itemize}
Then the generalised pair $(Z,B_Z+M_Z)$ given by the canonical bundle formula 
$$K_X+B\sim_{\R} f^*(K_Z+B_Z+M_Z)$$
is generalised $\delta$-lc along $V$.
\end{thm}
\begin{rem}
(1) When $d=1$, Theorem \ref{thm:main} was proved in \cite[Theorem 1.4]{Ch22} with an {\color{black} optimal} estimate for $\delta$. 

(2) When $f:X\to Z$ is a toric morphism between toric varieties and $(X,B)$ is a toric pair, Theorem \ref{thm:main} was proved in \cite[Theorem 1.8]{Ch23} with an effective estimate for $\delta$.

{\color{black}(3) When $V$ is a point, the problem was mentioned to us by Florin Ambro.}
\end{rem}

As a corollary, we have the following global version of Theorem \ref{thm:main}.
\begin{cor}\label{cor:vertical}
Let $d\in \N$ and let $\epsilon\in \R^{>0}$. Then there is $\delta\in \R^{>0}$ depending only on $d,\epsilon$ satisfying the following. Let $(X,B)$ be a pair and $f:X\to Z$ be a contraction such that 
\begin{itemize}
    \item $\dim X-\dim Z=d$,
    \item $K_X+B\sim_{\R} 0/Z$,
    \item $X$ is of Fano type over $Z$, and
    \item $(X,B)$ is {\color{black}$\epsilon$-lc vertically over $Z$}.
\end{itemize}
Then the generalised pair $(Z,B_Z+M_Z)$ given by the canonical bundle formula 
  $$K_X+B\sim_{\R} f^*(K_Z+B_Z+M_Z)$$
  is generalised $\delta$-lc.
  %that is, the coefficients of the discriminant b-divisor ${\bf B}_Z$ are $\leq 1-\delta$.
\end{cor}

As a consequence, we obtain the boundedness of singularities on base spaces and the boundedness of multiplicities of fibres over codimension one points for vertical $\epsilon$-lc Fano fibrations.

\begin{cor}\label{cor:sing-mult}
Let $d\in \N$ and let $\epsilon\in \R^{>0}$. Then there is $\delta\in \R^{>0}$ depending only on $d,\epsilon$ satisfying the following. Let $f:X\to Z$ be a contraction such that 
{\color{black}
\begin{itemize}
\item $\dim X-\dim Z=d$
\item $-K_X$ is ample over $Z$
\item $X$ is klt, and
\item $X$ is $\epsilon$-lc vertically over $Z$.
\end{itemize}
Then we have}

(1) $Z$ is $\delta$-lc if $K_Z$ is $\Q$-Cartier, and 

(2) for any codimension one point $z\in Z$, the multiplicity of each component of $f^*z$ is bounded from above by $1/\delta$.

\end{cor}

\subsection*{Acknowledgements} 
The first author {\color{black} was} supported by a grant from Tsinghua University and a grant of the National Program of Overseas High Level Talent. The second author {\color{black} was} supported by the start-up fund from  Sun Yat-sen University.

\section{Preliminaries}
We will freely use the standard notations and definitions in \cite{KM98,BCHM10}. 

A \emph{contraction} is a projective morphism $f:X\rightarrow Z$ with $f_*\mathcal{O}_X=\mathcal{O}_Z$.

Let $X$ be a variety and let $D$ be an $\R$-divisor on $X$. For a prime divisor $E$ on $X$, by $\mu_E D$ we mean the coefficient of $E$ in $D$. % If $D$ is $\R$-Cartier and $T$ is a prime divisor over $X$, i.e., on some birational model $f:Y\to X$, then by $\mu_T D$ we mean $\mu_T f^*D$

\subsection{Pairs and singularities} A \emph{sub-pair} $(X,B)$ consists of a normal variety $X$ and an $\R$-divisor $B$ on $X$ such that $K_X+B$ is $\R$-Cartier. A sub-pair $(X,B)$ is called a \emph{pair} if $B$ is effective. 

Let $(X,B)$ be a sub-pair and $E$ be a prime divisor on some birational model $\phi:X'\to X$. The \emph{center} of $E$ on $X$ is defined to be the image of $E$ on $X$ under the morphism $\phi$.
We may write $K_{X'}+B'=\phi^*(K_X+B)$ for some uniquely determined $B'$. Then the \emph{log discrepancy} $a(E,X,B)$ of $E$ with respect to $(X,B)$ is defined as $1-\mu_E B'$.

\begin{defn}
Let $\epsilon$ be a non-negative real number, $(X,B)$ be a (sub-)pair and $V\subseteq X$ be a closed subvariety. We say $(X,B)$ is  \emph{(sub-)$\epsilon$-lc} (resp.  \emph{(sub-)lc},  \emph{(sub-)klt}) along $V$ if 
$$a(E,X,B)\geq \epsilon ~(\text{resp. $\geq 0$, $>0$})$$
for any prime divisor $E$ over $X$ with $\cent_X E\subseteq V$. In the case when $V=X$, we say that $(X,B)$ is  (sub-)$\epsilon$-lc (resp.  (sub-)lc,  (sub-)klt).
\end{defn}

\subsection{b-divisors}
A \emph{b-$\R$-Cartier b-divisor} over a variety $X$ is the choice of a projective birational morphism $Y\to X$ from a normal variety and an $\R$-Cartier $\R$-divisor $M$ on $Y$ up to the following equivalence: another projective birational morphism $Y'\to X$ and an $\R$-Cartier $\R$-divisor $M'$ define the same b-$\R$-Cartier b-divisor if there is a common resolution $W\to Y$ and $W\to Y'$ on which the pullbacks of $M$ and $M'$ coincide.

\subsection{Generalised pairs}
%We will follow the original definitions in \cite{BZ16} and adopt the notations in \cite{HL21}. 
%Notice that there is a small difference in this paper: all generalized (sub-)pairs are assumed to be NQC unless stated otherwise.
A \emph{generalised sub-pair} consists of 
\begin{itemize}
\item a normal variety $X$ with a projective morphism $X\to Z$,
\item an $\R$-divisor $B$ on $X$, and 
\item a b-$\R$-Cartier b-divisor over $X$ represented by some projective birational morphism $X'\xrightarrow{\phi} X$ and an $\R$-Cartier $\R$-divisor $M'$ on $X'$
\end{itemize}
such that $M'$ is nef over $Z$ and $K_X+B+M$ is $\R$-Cartier, where $M:=\phi_{*} M'$. In the case when $B\geq 0$, we call it a \emph{generalised pair}.

Since a b-$\R$-Cartier b-divisor is defined birationally, in practice we will often replace $X'$ with a resolution and replace $M'$ with its pullback.

Let $(X,B+M)$ be a generalised sub-pair with data $X'\xrightarrow{\phi} X\to Z$ and $M'$. Let  $E$ be a prime divisor over $X$. Replacing $X'$ we may assume that $E$ is a prime divisor on $X'$. We can write
$$K_{X'}+B'+M'=\phi^*(K_X+B+M)$$
for some uniquely determined $B'$. Then the log discrepancy of $E$ with respect to $(X,B+M)$ is defined as $1-\mu_E B'$ and denoted by $a(E,X,B+M)$.

\begin{defn}
  Let $\epsilon$ be a non-negative real number, $(X,B+M)$ be a generalised (sub-)pair and $V\subseteq X$ be a closed subvariety. We say $(X,B+M)$ is \emph{generalised  (sub-)$\epsilon$-lc} (resp.  \emph{(sub-)lc},  \emph{(sub-)klt}) along $V$ if
$$a(E,X,B+M)\geq \epsilon ~(\text{resp. $\geq 0$, $>0$})$$
for any prime divisor $E$ over $X$ with $\cent_X E\subseteq V$. In the case when $V=X$, we say that $(X,B+M)$ is generalised  (sub-)$\epsilon$-lc (resp.  (sub-)lc,  (sub-)klt).
\end{defn}

\begin{lem}{\rm (cf. \cite[Lemma 2.7]{Ch23})}\label{lem:lc}
Let $(X,B+M)$ be a generalised pair with data $Y\xrightarrow{\pi} X\to Z$ and $M_Y$ on $Y$. If $(X,B+M)$ is generalised lc along a closed subvariety $V\subseteq X$, then it is generalised lc in a neighbourhood of $V$.
  \end{lem}
  \begin{proof}
  Assume the contrary that $(X,B+M)$ is not lc in any neighborhood of $V$. Then there is a prime divisor $D$ over $X$ such that ${\color{black}(\cent_X D)}\cap V\neq \emptyset$ and $a(D,X,B+M)< 0$.
Replace $Y$ with a higher resolution such that
\begin{itemize}
      \item $D$ is a prime divisor on $Y$,
      \item $\pi^{-1}V$ is a divisor on $Y$, say $F$, and 
      \item $D+F$ is a simple normal crossing divisor on $Y$.
\end{itemize}
Let $K_Y+B_Y+M_Y$ be the pullback of $K_X+B+M$ to $Y$. Since ${\color{black}(\cent_X D)}\cap V\neq \emptyset$, there is a component $E$ of $F$ such that $D\cap E\neq \emptyset$. Denote $d=\mu_D B_Y>1$ and $e=\mu_E B_Y$. Blowing up $D\cap E$, we get a new resolution $\pi':Y'\rightarrow X$. Let $K_{Y'}+B_{Y'}$ be the pullback of $K_{Y}+B_Y$, then for any prime divisor $L$ on $Y'$ we have $a(L,X,B+M)=1-\mu_L B_{Y'}$.

Denote by $E'$ the exceptional$/Y$ divisor on $Y'$. Then we have $\cent_X E'\subseteq V$, $\mu_{E'} B_{Y'}\geq e+d-1>e$ and $E'$ meets $D'$ transversely, where $D'$ is the birational transformation of $D$ on $Y'$. So, by successively blowing up, we eventually obtain a prime divisor $\widetilde{E}$ over $X$ such that $\cent_X \widetilde{E}\subseteq V$ and $a(\widetilde{E},X,B+M)<0$, which is in contradiction with {\color{black}the assumption} that $(X,B+M)$ is  generalised lc along $V$.
\end{proof}

\subsection{Generalised canonical bundle formula}
The canonical bundle formula for usual pairs was established by Kawamata \cite{Ka97,Ka98} and Ambro \cite{Am05}. In this subsection we will briefly introduce the canonical bundle formula for generalised pairs. For more details, we refer the readers to \cite{Fi20}, \cite{JLX22} and \cite[\S 11.4]{CHLX23}.

Assume that 
\begin{itemize}
\item $(X,B+M)$ is a generalised pair with data $X'\rightarrow X\rightarrow {\color{black}U}$ and $M'$,
\item $f:X\to Z{\color{black}/U}$ is a contraction {\color{black}between varieties over $U$} with $\dim Z>0$,
\item $(X,B+M)$ is generalised lc over the generic point of $Z$, and
\item $K_X+B+M\sim_{\R} 0/Z$.
\end{itemize}

For any prime divisor $D$ on $Z$, let $t_D$ be the largest real number such that $(X,B+tf^*D+M)$ is generalised lc over the generic point of $D$. This {\color{black} makes} sense even if $D$ is not $\Q$-Cartier because we only need the pullback of $D$ over the generic point of $D$ where $Z$ is smooth.
We then set $B_Z=\sum_D (1-t_D)D$ where $D$ runs over all prime divisors on $Z$. Having defined $B_Z$, we can find $M_Z$ such that
$$K_X+B+M\sim_{\R} f^*(K_Z+B_Z+M_Z),$$
%Then $B_Z$ is an effective divisor.
where $M_Z$ is determined up to $\R$-linear equivalence. We call $B_Z$ the \emph{discriminant divisor} and call $M_Z$ the \emph{moduli divisor}.

For any birational morphism $\sigma: Z'\to Z$ from a normal variety, we can similarly defines $B_{Z'}$ and $M_{Z'}$ so that $\sigma_* B_{Z'}=B_Z$ and $\sigma_* M_{Z'}=M_Z$. Putting all the $B_{Z'}$ (resp. $M_{Z'}$) together for all the possible $Z'$ determines a b-divisor $\textbf{B}_Z$ (resp.  $\textbf{M}_Z$), which is called the \emph{discriminant} (resp. \emph{moduli}) \emph{b-divisor}.
  
It was shown in \cite[Theorem 11.4.4]{CHLX23} that taking $Z'$ high enough, $M_{Z'}$ is nef over {\color{black}$U$} and for any other resolution $\pi: Z''\to Z'$ we have $M_{Z''}=\pi^*M_{Z'}$. Hence we can regard $(Z,B_Z+M_Z)$ as a generalised pair with data $Z'\to Z {\color{black}\to U}$ and $M_{Z'}$. We call it the generalised pair given by the canonical bundle formula of $f:(X,B+M)\rightarrow Z$.

\subsection{Fano type contractions} Let $X\to Z$ be a contraction. 
We say $X$ is \emph{of Fano type over} $Z$ if there is a klt pair $(X,C)$ such that $-(K_X+C)$ is ample over $Z$. This is equivalent to having a klt pair $(X,B)$ such that $K_X+B\sim_{\R} 0/Z$ and $-K_X$ is big over $Z$. 

If $X$ is of Fano type over $Z$, then $X$ is Mori dream space over $Z$ by \cite{BCHM10}. That is, for any $\R$-Cartier $\R$-divisor $D$ on $X$, one can run an MMP on $D$ over $Z$ and the MMP ends with a good minimal model or a Mori fiber space over $Z$.

\begin{lem}{\rm(\cite[Lemma 2.8]{PS09})}\label{lem:preserve}
Suppose that $X$ is of Fano type over $Z$.

{\rm{(1)}} Suppose $X \dashrightarrow X'/Z$ is a birational map whose inverse does not contract any divisor. Then $X'$ is of Fano type over $Z$.

{\rm{(2)}} Suppose $X\to Y/Z$ is a contraction of varieties over $Z$. Then $Y$ is of Fano type over $Z$.

{\rm{(3)}} Let $(X,B)$ be an lc pair such that $-(K_X+B)$ is nef over $Z$. Let $Y$ be a normal variety with a projective birational morphism $Y\to X$ and let $K_Y+B_Y$ be the pullback of $K_X+B$. If every exceptional$/X$ component of $B_Y$ has positive coefficient, then $Y$ is of Fano type over $Z$.
\end{lem}

\iffalse
\begin{lem}\label{lem:klt}
Suppose that $X$ is of Fano type over $Z$. Then there is a klt pair on $Z$. 
\end{lem}
\begin{proof}
As $X$ is of Fano type over $Z$, by \cite[Lemma-Definition 2.6]{PS09}, there is a $\Q$-boundary $B$ on $X$ such that $(X,B)$ is a klt pair and $K_X+B\sim_{\Q} 0/Z$. The lemma follows from \cite[Theorem 0.2]{Am05}.
\end{proof}
\fi

\section{Proof of main results}
In this section we will prove a more general form of Theorem \ref{thm:main} for generalised pairs.

\begin{thm}\label{thm:generalised}
  Let $d\in \N$ and let $\epsilon\in \R^{>0}$. Then there is $\delta\in \R^{>0}$ depending only on $d,\epsilon$ satisfying the following. Let $(X,B+M)$ be a generalised pair with data $X'\xrightarrow{\phi} X\xrightarrow{f} Z$ and $M'$ where $f:X\to Z$ is a contraction and let $V\subseteq Z$ be a closed subvariety such that 
  \begin{itemize}
    \item $\dim X-\dim Z=d$,
    \item $K_X+B+M\sim_{\R} 0/Z$,
    \item $X$ is of Fano type over $Z$, and
    \item $(X,B+M)$ is generalised $\epsilon$-lc along $f^{-1} V.$
  \end{itemize}
  Then the generalised pair $(Z,B_Z+M_Z)$ given by the canonical bundle formula 
  $$K_X+B+M\sim_{\R} f^*(K_Z+B_Z+M_Z)$$
  is generalised $\delta$-lc along $V$.
\end{thm}

We reduce Theorem \ref{thm:generalised} to Theorem \ref{thm:main}.
\begin{lem}\label{lem:pure-to-generalised}
Let $d\in \N$. Assume Theorem \ref{thm:main} holds in relative dimension $d$. Then Theorem \ref{thm:generalised} holds in relative dimension $d$.
\end{lem}
\begin{proof}
Shrinking $Z$ around $V$, by Lemma \ref{lem:lc} we may assume that $(X,B+M)$ is generalised lc. By \cite[Lemma 4.5]{BZ16}, there exists a normal variety $Y$ with a projective birational morphism $Y\to X$ so that $Y$ is $\Q$-factorial and each exceptional$/X$ prime divisor on $Y$ has log discrepancy $=0$ with respect to $(X,B+M)$. Replacing $X$ by $Y$, we may assume that $X$ is $\Q$-factorial.  

By running an MMP on $B+M$ over $Z$, we may assume that $B+M$ is big and nef over $Z$. Pick any resolution $Z'\to Z$ and a prime divisor $D$ on $Z'$ with $\cent_Z D\subseteq V$. Replacing $X'$ with a higher resolution, we can assume that the induced map $f':X'\dashrightarrow Z'$ is a morphism. Let $K_{X'}+B'+M'$ be the pullback of $K_X+B+M$ to $X'$. Then $(X',B')$ is sub-$\epsilon$-lc along $\phi^{-1}f^{-1} V$. We need to show that the lc threshold of $f'^*D$ with respect to $(X',B')$ over the generic point of $D$, say $t$, is bounded from below away from zero.

Pick $s\in (0,1)$ sufficiently close to 1. We can write
$$K_{X'}+B'_s+sM'=\phi^*(K_X+sB+sM).$$
Then $B'-B'_s\geq 0$ has sufficiently small coefficients. As $X$ is of Fano type over $Z$, $X$ is klt, so $(X,sB+sM)$ is generalised klt. It follows that $(X',B'_s)$ is sub-klt. Moreover, $(X',B'_s)$ is sub-$\epsilon$-lc along $\phi^{-1}f^{-1} V$. As $B+M$ is big and nef over $Z$ and as $M'$ is nef over $Z$, we can find 
$$0\leq L\sim_{\R} sM'+(1-s)\phi^*(B+M)/Z$$
such that $(X',\Delta':=B'_s+L)$ is sub-$\frac{\epsilon}{2}$-lc along $\phi^{-1}f^{-1} V$ and is sub-klt. Moreover, we have 
\begin{itemize}
\item $K_{X'}+\Delta'\sim_{\R}K_{X'}+B'+M'\sim_{\R} 0/Z$, and 
\item $\Delta'\geq B'_s$, in particular $\Delta\geq sB\geq 0$ where $\Delta$ is the pushdown of $\Delta'$ to $X$.
\end{itemize}
Therefore, $(X,\Delta)$ is $\frac{\epsilon}{2}$-lc along $f^{-1}V$ and $K_X+\Delta\sim_{\R} 0/Z$. Applying Theorem \ref{thm:main} to $(X,\Delta)\to Z$, we deduce that the lc threshold of $f'^*D$ with respect to $(X',\Delta')$ over the generic point of $D$, say $u$, is bounded from below by a fixed positive number depending only on $d,\epsilon$. By construction, we have either $u\leq t$ or $u>t$ but $u-t$ is sufficiently small, which implies that $t$ is also bounded from below away from zero.
\end{proof}

To prove Theorem \ref{thm:main}, we need a couple of lemmas.

\begin{lem}\label{lem:factor}
Let $d\in \N$. Assume Theorem \ref{thm:main} holds in relative dimension $<d$. Then the theorem holds in relative dimension $d$  when $f:X\to Z$ can be factored as $X\xrightarrow{h}Y \xrightarrow{g}Z$ where $h,g$ are contractions of relative dimension $<d$.
\end{lem}
\begin{proof}
{\color{black} As $X$ is of Fano type over $Z$, there is a klt pair $(X,\Gamma)$ such that $K_X+\Gamma\sim_{\R} 0/Z$ and $-K_X$ is big over $Z$, so $K_X+\Gamma\sim_{\R} 0/Y$ and $-K_X$ is big over $Y$. It follows that $X$ is of Fano type over $Y$.} By assumption, if we apply the canonical bundle formula
  $$K_{X}+B \sim_{\R} h^*(K_{Y}+B_{Y}+M_{Y}),$$
then $(Y,B_{Y}+M_{Y})$ is generalised $\delta'$-lc along $g^{-1}V$ for some fixed $\delta'>0$. 

{\color{black} By Lemma \ref{lem:preserve} (2), $Y$ is of Fano type over $Z$.}
By Lemma \ref{lem:pure-to-generalised}, we can apply Theorem \ref{thm:generalised} to
$$(Y,B_{Y}+M_{Y})\to Z,$$ to deduce that $(Z,B_Z+M_Z)$ is  generalised $\delta$-lc along $V$ for some fixed $\delta>0$, where $(Z,B_Z+M_Z)$ is the generalised pair given by the canonical bundle formula
$$K_Y+B_Y+M_Y\sim_{\R} g^*(K_{Z}+B_{Z}+M_{Z}).$$

By \cite[Lemma 2.1]{BC21}, $(Z,B_Z+M_Z)$ is also the generalised pair given by the canonical bundle formula of $(X,B)\to Z$. So we are done.
\end{proof}
\begin{lem}\label{lem:base-dimension-1}
Assume Theorem \ref{thm:main} holds when $\dim Z=1$. Then the theorem holds in general.
\end{lem}
\begin{proof}
We prove the theorem by induction on the dimension of $Z$. Suppose the theorem holds when $\dim Z\leq n-1$ for some $n\geq 2$. We will show that the theorem holds when $\dim Z=n$. The case when $V=Z$ was proved in \cite[Theorem 1.1]{Bi23}, so we may assume that $V\neq Z$. Shrinking $Z$ around $V$, by Lemma \ref{lem:lc} we may assume that $(X,B)$ is lc. Then $(Z,B_Z+M_Z)$ is generalised lc.

Next we reduce the theorem to the case that $Z$ is $\Q$-factorial. Replacing $(X,B)$ by a $\Q$-factorial dlt model, we may assume that $X$ is $\Q$-factorial. Run an MMP on $K_X$ over $Z$ which ends with a Mori fiber space $\hat{X}\to \hat{Z}$. Then $\hat{Z}$ is  $\Q$-factorial. If $\dim \hat{Z}>\dim Z$, by induction on relative dimension, the theorem follows from Lemma \ref{lem:factor}. So we can assume that $\hat{Z}\to Z$ is birational. Replacing $X$ with $\hat{X}$ and $Z$ with $\hat{Z}$, we may assume that $Z$ is $\Q$-factorial.

Let $D$ be a prime divisor over $Z$ such that $\cent_Z D\subseteq V$ and $a(D,Z,B_Z+M_Z)<1$. We need to show $a(D,Z,B_Z+M_Z)$ is bounded from below by a positive constant depending only on $d,\epsilon$. We divide the rest of the proof in two cases.

\textbf{Case 1: $D$ is a prime divisor on $Z$.} It suffices to show that the lc threshold of $f^*D$ over the generic point of $D$ with respect to the pair $(X,B)$ is bounded from below by a fixed positive number depending only on $d,\epsilon$. %By removing some codimension two closed subset of $Z$ we may assume that $Z$ is smooth.

Pick a general hyperplane section $H\subseteq Z$ (which would intersect $D$ as $\dim Z>1$) and let $G=f^*H$. This ensures that  $(X,B+G)$ is $\epsilon$-lc along $f^{-1}V$ (here we use the assumption that $V\neq Z$).
Letting
$$K_G+B_G=(K_X+B+G)|_G,$$
we get a pair $(G,B_G)$ and a Fano type contraction $g:G\to H$ such that $K_G+B_G\sim_{\R} 0/H$ and $(G,B_G)$ is $\epsilon$-lc along $g^{-1}(V\cap H)$.  

By removing some codimension two closed subset of $Z$, we may assume that $D_H:=D\cap H$ is irreducible. Let $t$ be the lc threshold of $g^*D_H$ over the generic point of $D_H$ with respect to the pair $(G,B_G)$. Since $\dim H=\dim Z-1=n-1$, by the inductive hypothesis, $t$ is bounded from below by a fixed positive number depending only on $d,\epsilon$. By further shrinking $Z$, we can assume that $(G,B_G+tg^*D_H)$ is lc. Applying inversion of adjunction \cite{Ka07}, we deduce that $(X,B+G+tf^*D)$ is lc in a neighbourhood of $G=f^*H$. Therefore, $(X,B+tf^*D)$ is lc over the generic point of $D$ and we are done.

\textbf{Case 2: $D$ is exceptional over $Z$.} We claim that there is a birational morphism $\pi:Z''\to Z$ extracting $D$ with $\text{Exc}(\pi)=D$. {\color{black}Indeed, as $X$ is of Fano type over $Z$, by Lemma \ref{lem:preserve} (2) $Z$ is of Fano type over $Z$, hence there is a klt pair $(Z,\Gamma)$ on $Z$.} Pick sufficiently small $s>0$. Then  the claim follows from the following three facts:
\begin{itemize}
  \item $(Z,s\Gamma+(1-s)B_Z+(1-s)M_Z)$ is generalised klt, 
  \item $a(D,Z,s\Gamma+(1-s)B_Z+(1-s)M_Z)<1$ (as $a(D,Z,B_Z+M_Z)<1$), and
  \item $Z$ is $\Q$-factorial.
\end{itemize}
As $\text{Exc}(\pi)=D$ and as $\cent_Z D\subseteq V$, $\pi$ is {\color{black}an isomorphism} over $U:=Z\setminus V$.

Let $\phi:W\to X$ be a log resolution of $(X,B)$ such that the induced rational map $W\dashrightarrow Z''$ is a morphism. Define a boundary $\Delta_W$ on $W$ such that
$$\mu_E \Delta_W =
\begin{cases}
1 & \text{if $E$ is exceptional$/X$ and $f(\phi(E))\nsubseteq V$,} \\
1-\epsilon & \text{if $E$ is exceptional$/X$ and $f(\phi(E))\subseteq V$,}\\
\mu_E \widetilde{B}  & \text{if $E$ is not exceptional$/X$.}
\end{cases}
$$ 
for any prime divisor $E$ on $W$, where $\widetilde{B}$ is the birational transform of $B$ on $W$. Then $(W,\Delta_W)$ is $\epsilon$-lc along $\phi^{-1}f^{-1}V$ and we can write
$$K_W+\Delta_W= \phi^*(K_X+B)+C,$$
where $C\geq 0$ and $C$ is exceptional$/X$. In particular, by \cite[Theorem 1.8]{Bi12} if we run an MMP$/X$ on $K_W+\Delta_W$ with scaling of some ample$/X$ divisor, then the MMP terminates with a model $(W',\Delta_{W'})$ such that $K_{W'}+\Delta_{W'}$ is the pullback of $K_X+B$.

Let $G$ be the graph of {\color{black} the} rational map $X\dashrightarrow Z''$, i.e. the closure in $X\times Z''$ of the graph of $X_0\to Z''$ where $X_0\subseteq X$ is the domain of $X\dashrightarrow Z''$. Since $W$ maps to both $X$ and $Z''$, we get an induced morphism $W\to G$. Running an MMP$/G$ on $K_W+\Delta_W$ with scaling of some ample$/G$ divisor, we end up with a model $Y$ on which $K_Y+\Delta_Y$ is nef$/G$. As $\pi:Z''\to Z$ is {\color{black}an isomorphism} over $U$, $f^{-1} U\subseteq X_0$ and hence ${\color{black}G}\to X$ is an isomorphism over $U$. So, $K_Y+\Delta_Y$ is the pullback of $K_X+B$ over $U$ by the last paragraph. Since $K_X+B\sim_{\R} 0/Z$, $K_Y+\Delta_Y\sim_{\R} 0$ over $U$. As $(W,\Delta_W)$ is $\epsilon$-lc along {\color{black} the} inverse image of $V$, so is $(Y,\Delta_Y)$. So, by \cite[Theorem 1.4]{Bi12} and \cite{HMX14}, we can run an MMP$/Z''$ on $K_Y+\Delta_Y$ which ends up with a model $X'$ on which $K_{X'}+\Delta_{X'}$ is semi-ample$/Z''$. Since $K_Y+\Delta_Y\sim_{\R} 0$ over $U$, the MMP does not modify $Y$ over $U$, i.e. $Y\dashrightarrow X'$ is an isomorphism over $U$.

Let $f':X'\to Z'$ be the contraction$/Z''$ defined by $K_{X'}+\Delta_{X'}$. By construction, the induced morphism $Z'\to Z''$ is {\color{black}an isomorphism} over $U$, hence is birational. Moreover, $(X',\Delta_{X'})$ is $\epsilon$-lc along the inverse image of $V$ and $K_{X'}+\Delta_{X'}\sim_{\R} 0/Z'$. 

Next we show that $X'$ is of Fano type over $Z'$. By construction, over $U$ we have $X'\dashrightarrow X$ is a morphism, $K_{X'}+\Delta_{X'}$ is the pullback of $K_X+B$ and every exceptional$/X$ component of $\Delta_{X'}$ has coefficient $=1$. So, by Lemma \ref{lem:preserve} (3), $X'$ is of Fano type over $U$, which implies that $-K_{X'}$ is big over $Z'$. As $X$ is of Fano type over $Z$, there is a klt pair $(X,L)$ such that $K_X+L\sim_{\R} 0/Z$. Let $K_{X'}+L_{X'}$ be the {\color{black}crepant} pullback of $K_X+L$ to $X'$ (i.e. pull  back  $K_X+L$ to a common resolution of $X$ and $X'$ and then push it down to $X'$). Then $(X',L_{X'})$ is sub-klt and $K_{X'}+L_{X'}\sim_{\R} 0/Z'$. Note that every exceptional$/X$ component of $\Delta_{X'}$ has positive coefficient (here we may assume that $\epsilon<1$). Put
$$\varXi_{X'}:=(1-s)\Delta_{X'}+sL_{X'}$$
for sufficiently small positive $s$. Then  $(X',\varXi_{X'})$ is a klt pair and $K_{X'}+\varXi_{X'}\sim_{\R} 0/Z'$. This implies that $X'$ is of Fano type over $Z'$.

Therefore, $(X',\Delta_{X'})$ and $f':X'\to Z'$ satisfy the assumptions of Theorem \ref{thm:main}. We can view $D$ as a prime divisor on $Z'$. Let $t$ be the lc threshold of $f'^*D$ over the generic point of $D$ with respect to $(X',\Delta_{X'})$. By the argument in Case 1, $t$ is bounded from below by a fixed positive number depending only on $d,\epsilon$.

Let $p:V\to X$ and $q:V\to X'$ be a common resolution and let
$$M:=K_{X'}+\Delta_{X'}-q_*p^*(K_X+B).$$
As mentioned above we have
$$K_W+\Delta_W-\phi^*(K_X+B)=C\geq 0.$$
So $M\geq 0$ as $M$ is just the pushdown of $K_W+\Delta_W-\phi^*(K_X+B)$ via $W\dashrightarrow X'$. It follows from $K_X+B\sim_{\R} 0/Z$ that
$$p^*(K_X+B)=q^*q_*p^* (K_X+B).$$
{\color{black} Combined with} $M\geq 0$ we get
$$q^*(K_{X'}+\Delta_{X'})-p^*(K_X+B)=q^*M\geq 0.$$
Write
$$K_V+B_V=p^*(K_X+B) \quad \text{ and } \quad K_V+\Delta_V=q^*(K_{X'}+\Delta_{X'}).$$ 
Then $B_V\leq \Delta_V$. Moreover, $(V,\Delta_V+tq^*f'^*D)$ is sub-lc over the generic point of $D$. Therefore, $(V,B_V+tq^*f'^*D)$ is sub-lc over the generic point of $D$ and we are done.
\end{proof}
\begin{lem}\label{lem:epsilon-lc}
Theorem \ref{thm:main} holds when $\dim Z=1$, $X$ is $\epsilon$-lc and $-K_X$ is big and nef over $Z$.
\end{lem}
\begin{proof}
%We can assume that $\epsilon$ is sufficiently small and that it is of the form $\frac{1}{p}$ for some $p\in \N$. 
We may assume that $\epsilon< 1$. When $V=Z$, the theorem follows from \cite[Theorem 1.1]{Bi23}. So we may assume that $V=\{z\}$, where $z$ is a closed point of $Z$. Shrinking $Z$ around $z$, we may assume that $(X,B)$ is lc by Lemma \ref{lem:lc}.

Let $t$ be the largest number such that $(X,B+tf^*z)$ is $\frac{\epsilon}{2}$-lc along $f^{-1}\{z\}$. It is enough to show that $t$ is bounded from below away from zero. We can find a prime divisor $T$ over $X$ such that $f(\cent_X T)=\{z\}$ and 
$$a(T,X,B+tf^*z)=\frac{\epsilon}{2}.$$
Since $(X,B)$ is $\epsilon$-lc along $f^{-1}\{z\}$, we have
$$t\cdot \mu_T f^*z\geq \frac{\epsilon}{2}.$$
So it is enough to show that $\mu_T f^*z$ is bounded from above. Replacing $\epsilon$ by $\frac{\epsilon}{2}$ and $B$ by $B+tf^*z$, we can assume that $a(T,X,B)=\epsilon$.

{\color{black}
Let $Y\to X$ be the birational contraction which extracts $T$ and such that $Y$ is $\Q$-factorial. Let $K_Y+B_Y$ be the pullback of $K_X+B$ to $Y$. By Lemme \ref{lem:preserve} (3), $Y$ is of Fano type over $Z$. Run an MMP on $-K_Y$ over $Z$ which ends with a good minimal model $Y'$, that is, $-K_{Y'}$ is semiample over $Z$. By Lemma \ref{lem:preserve} (1), $Y'$ is of Fano type over $Z$, so $Y'$ is klt. We claim that $Y'$ is $\epsilon$-lc over $Z\setminus\{z\}$. Indeed, as $-K_X$ is nef over $Z$ and as $Y$ is isomorphic to $X$ over $Z\setminus\{z\}$, $-K_Y$ is nef over $Z\setminus\{z\}$, hence the MMP does not modify $Y$ over $Z\setminus\{z\}$. It follows that $Y'$ is isomorphic to $X$ over $Z\setminus\{z\}$. So
$Y'$ is $\epsilon$-lc over $Z\setminus\{z\}$ as $X$ is $\epsilon$-lc.

Take a general $0\leq R_{Y'}\sim_{\Q} -K_{Y'}/Z$. Then $(Y',R_{Y'})$ is klt globally and is $\epsilon$-lc over $Z\setminus\{z\}$, as $-K_{Y'}$ is semiample over $Z$. Let $K_Y+R_Y$ be the crepant pullback of $K_{Y'}+R_{Y'}$ to $Y$, then $K_Y+R_Y\sim_{\Q} 0/Z$. As $Y\dashrightarrow Y'$ is an MMP on $-K_Y$, we have $R_{Y}\geq 0$. Moreover, we have $(Y,R_{Y})$ is klt globally and is $\epsilon$-lc over $Z\setminus\{z\}$. 

Pushing down $K_Y+R_Y$ to $X$, we get a klt pair $(X,R)$ which is $\epsilon$-lc over $Z\setminus \{z\}$. Moreover, we have $K_X+R\sim_{\Q} 0/Z$ and  $a(T,X,R)\leq 1$. Now $(X,\frac{1}{2} R+ \frac{1}{2}B)$ is $\frac{\epsilon}{2}$-lc: if $D$ is a prime divisor over $X$ which is horizontal over $Z$, then we use the fact that $(X,R)$ is $\epsilon$-lc over $Z\setminus\{z\}$; and if $D$ maps to $z$, we use the assumption that $(X,B)$ is $\epsilon$-lc along $f^{-1}\{z\}$. Moreover,
\begin{align*}
a(T,X,\frac{1}{2}R+\frac{1}{2}B)&=\frac{1}{2}a(T,X,R)+\frac{1}{2}a(T,X,B)%\\&=\frac{1}{2}a(T,X,B)
\leq \frac{1+\epsilon}{2}\leq 1.
\end{align*}
Therefore applying \cite[Theorem 1.1]{Bi23} we deduce that $\mu_T f^*z$ is bounded from above as required.
}
\end{proof}
\begin{proof}[Proof of Theorem \ref{thm:main}]
We may assume that $\epsilon\leq 1$. By Lemma \ref{lem:base-dimension-1}, we may assume that $\dim Z=1$. 
Shrinking $Z$ around  $V$, we may assume that $(X,B)$ is lc by Lemma \ref{lem:lc}.
Replacing $(X,B)$ by a $\Q$-factorial dlt model, we may assume that $X$ is $\Q$-factorial and $X$ is klt.

Let $Y\to X$ be the birational contraction which extracts all {\color{black} the} prime divisors $D$ with $a(D,X,0)<\epsilon$ and such that $Y$ is $\Q$-factorial. Then $Y$ is $\epsilon$-lc. Let $K_Y+B_Y$ be the pullback of $K_X+B$ to $Y$.
Run an MMP on $K_Y$ over $Z$ which  ends with a Mori fiber space $h:Y'\to  Z'$ over $Z$. Denote by $B_{Y'}$ the pushdown of $B_Y$ to $Y'$ and by $g$ the induced morphism $Z'\to Z$. By construction, $(Y',B_{Y'})$ is $\epsilon$-lc along $h^{-1}g^{-1}V$. 

If $\dim Z'=\dim Z$, then $Z'=Z$ (as $\dim Z=1$) and we apply Lemma \ref{lem:epsilon-lc} to $(Y',B_{Y'})\to Z$, which implies the theorem for $(X,B)\to Z$. 
If $\dim Z'>\dim Z$, we apply induction on relative dimension and apply Lemma \ref{lem:factor} to $(Y',B_{Y'})\to Z$, which implies the theorem for $(X,B)\to Z$. So we are done.
\end{proof}

\begin{proof}[Proof of Theorem \ref{thm:generalised}]
It follows from Theorem \ref{thm:main} and Lemma \ref{lem:pure-to-generalised}.
\end{proof}

\begin{proof}[Proof of Corollary \ref{cor:vertical}]
For any prime divisor $E$ over $Z$, $V:=\cent_Z E$ is a proper subvariety of $Z$. Applying Theorem \ref{thm:main}, $(Z,B_Z+M_Z)$ is generalised $\delta$-lc along $V$ where $\delta>0$ depends only on $d,\epsilon$. It follows that $a(E,Z,B_Z+M_Z)\geq \delta$. Hence $(Z,B_Z+M_Z)$ is generalised $\delta$-lc.
\end{proof}

\begin{proof}[Proof of Corollary \ref{cor:sing-mult}]
{\color{black}As $X$ is klt and $-K_X$ is ample over $Z$, $X$ is of Fano type over $Z$.} We may assume that $\epsilon<1$. Take a large integer $N$ such that $1/N\leq 1-\epsilon$ and $-NK_X$ is very ample over $Z$. Let $H$ be a general effective divisor such that $H\sim -NK_X/Z$ and let $B=\frac{1}{N} H$. Then $K_X+B\sim_{\Q} 0/Z$ and $(X,B)$ is {\color{black} $\epsilon$-lc vertically over $Z$}. Applying Corollary \ref{cor:vertical}, the generalised {\color{black} pair} $(Z,B_Z+M_Z)$ given by the canonical bundle formula is generalised $\delta$-lc where $\delta>0$ depends only on $d,\epsilon$. It follows that if $K_Z$ is $\Q$-Cartier, then $Z$ is $\delta$-lc. Moreover, by Corollary \ref{cor:vertical} for any codimension one point $z\in Z$, $(X,B+\delta f^*\overline{z})$ is lc over $z$, which implies that the multiplicity of each component of $f^*z$ is bounded from above by $1/\delta$.
\end{proof}

%%%%%%%%%%%%%%%%%%%%%%%%%%%%%%%%%%%%%


\begin{thebibliography}{LONGEST}

%\bibitem[Am99]{Am99} F. Ambro, \emph{The Adjunction Conjecture and its applications.}  arXiv:math/9903060v3.
%\bibitem[Am04]{Am04} F. Ambro, \emph{Shokurov's boundary property,} J. Differential Geom. 67 (2004), no. 2, 229 -- 255.

%\bibitem[Bi18]{Bi18} C. Birkar; \emph{Log Calabi-Yau fibrations}, arXiv:1811.10709v2.

%\bibitem[Bi16]{Bi16} C. Birkar, \emph{Singularities on the base of a {F}ano type fibration}, J. Reine Angew. Math. 715 (2016), 125 -- 142.
\bibitem[Am05]{Am05} F.~Ambro, 
\textit{The moduli b-divisor of an lc-trivial fibration}, 
Compos. Math. \textbf{141} (2005), no. 2, 385 -- 403.

\bibitem[Am22]{Am22} F.~Ambro, \emph{On toric Fano fibrations}, arXiv:2212.13862v2.

\bibitem[Bi12]{Bi12} C. Birkar, \emph{Existence of log canonical flips and a special LMMP}, Publ. Math. Inst. Hautes \'Etudes Sci. \textbf{115} (2012), no. 1, 325 -- 368.

\bibitem[Bi19]{Bi19} C. Birkar, \emph{Anti-pluricanonical systems on Fano varieties}, Ann. of Math. \textbf{190} (2019), no. 2, 345 -- 463.
\bibitem[Bi21]{Bi21} C. Birkar, \emph{Singularities of linear systems and boundedness of Fano varieties}. Ann. of Math. \textbf{193} (2021), no. 2, 347 -- 405.

%\bibitem[Bi22]{Bi22} C. Birkar, \emph{Boundedness of Fano type fibrations}, to appear in Ann. Sci. ENS, arXiv:2209.08797.

\bibitem[Bi23]{Bi23} C. Birkar, \textit{Singularities on Fano fibrations and beyond}, arXiv:2305.18770.

\bibitem[BC21]{BC21} C.~Birkar and Y.~Chen, \emph{Singularities on toric fibrations}, Sb. Math. \textbf{212} (2021), no. 3, 20 -- 38.

\bibitem[BCHM10]{BCHM10} C. Birkar, P. Cascini, C. D. Hacon and J. M\textsuperscript{c}Kernan, \emph{Existence of minimal models for varieties of log general type}, J. Amer. Math. Soc. \textbf{23} (2010), no. 2, 405 -- 468.

\bibitem[BZ16]{BZ16} C. Birkar and D-Q. Zhang, \emph{Effectivity of Iitaka fibrations and pluricanonical systems of polarized pairs}, Publ. Math. Inst. Hautes \'Etudes Sci. \textbf{123} (2016), 283 -- 331.

\bibitem[Ch22]{Ch22} B. Chen, \emph{Optimal bound for singularities on Fano type fibrations of relative dimension one}, arXiv:2210.08469v3.

\bibitem[Ch23]{Ch23} B. Chen, \emph{Effective bound for singularities on toric fibrations}, arXiv:2311.00985v2.

%\bibitem[CH21]{CH21} G. Chen and J. Han, \emph{Boundedness of $(\epsilon,n)$-complements for surfaces,} Adv. Math. \textbf{383} (2021), Paper No. 107703, 40 pp.

\bibitem[CHLX23]{CHLX23} G. Chen, J. Han, J. Liu and L. Xie, \emph{Minimal model program for algebraically integrable
foliations and generalized pairs}, arXiv:2309.15823v2.

\bibitem[Fi20]{Fi20} S. Filipazzi, \emph{On a generalized canonical bundle formula and generalized adjunction}, Ann. Sc. Norm. Super.
Pisa Cl. Sci. (5) Vol. XXI (2020), 1187 -- 1221.

\bibitem[HJL22]{HJL22} J. Han, C. Jiang and Y. Luo, \emph{Shokurov's conjecture on conic bundles with canonical singularities}, Forum Math. Sigma \textbf{10} (2022), e38, 1 -- 24.


\bibitem[HMX14]{HMX14} C. D. Hacon, J. M\textsuperscript{c}Kernan and C. Xu, \emph{ACC for log canonical thresholds}, Ann. of Math. \textbf{180} (2014), no. 2, 523 -- 571.

\bibitem[JLX22]{JLX22} J. Jiao, J. Liu and L. Xie, \emph{On generalized lc pairs with b-log abundant nef part}, arXiv:2202.11256v2.

\bibitem[Ka97]{Ka97}  Y. Kawamata, \emph{Subadjunction of log canonical divisors for a variety of codimension 2}, Contemp. Math. \textbf{207} (1997), 79 -- 88.
\bibitem[Ka98]{Ka98} Y. Kawamata, \emph{Subadjunction of log canonical divisors, II}, Amer. J. Math. \textbf{120} (1998), 893 -- 899.

\bibitem[Ka07]{Ka07} M. Kawakita, \emph{Inversion of adjunction on log canonicity}, Invent. Math. \textbf{167} (2007), no. 1, 129 -- 133.

\bibitem[Ko93]{Ko93}  J. Koll\'ar, \emph{Effective base point freeness}, Math. Ann. \textbf{296} (1993), no. 4, 595 -- 605.


\bibitem[KM98]{KM98} J. Koll\'{a}r and S. Mori, \emph{Birational geometry of algebraic varieties}, Cambridge Tracts in Math. \textbf{134}, Cambridge Univ. Press, 1998.

\bibitem[PS09]{PS09} Y.~G.~Prokhorov and V.~V.~Shokurov, \textit{Towards the second main theorem on complements}, J. Algebraic Geom. \textbf{18} (2009), no. 1, 151 -- 199.




%\bibitem[HMX14]{HMX14} C. D. Hacon, J. McKernan and C. Xu, \emph{ACC for log canonical thresholds}, Ann. of Math. (2) 180 (2014),no. 2, 523 -- 571.


%\bibitem[Ka15]{Ka15} M. Kawakita, \emph{The index of a threefold canonical singularity}, Amer. J. Math. 137 (2015), no. 1, 271 -- 280.

%\bibitem[Ko$^+$92]{Ko92} J. Koll\'ar \'et al., \emph{Flip and abundance for algebraic threefolds}, As\'erisque No. 211, 1992.


%\bibitem[Ma02]{Ma02} K. Matsuki, \emph{Introduction to the Mori program}, Universitext. Springer-Verlag, New York, 2002. xxiv+478 pp.


%\bibitem[PS01]{PS01}  Y. G. Prokhorov and V. V. Shokurov, \emph{The first fundamental theorem on complements: from global to local.} Izv. Ross. Akad. Nauk Ser. Mat. 65 (2001), no.6, 99 -- 128.
 

%\bibitem[Sh92]{Sh92} V. V. Shokurov, \emph{Threefold log flips}, with an appendix in English by Y. Kawamata, Izv. Ross. Akad. Nauk Ser. Mat. 56 (1992), no. 1, 105 -- 203.

\end{thebibliography}
\end{document}